\documentclass[a4paper, 11pt]{amsart}

\usepackage{subcaption}
\usepackage{pgfplots}
\usepgfplotslibrary{external}
\usepackage[linesnumbered,commentsnumbered,ruled,vlined]{algorithm2e}

\newlength\mylen
\newcommand\myinput[1]{%
  \settowidth\mylen{\KwIn{}}%
  \setlength\hangindent{\mylen}%
  \hspace*{\mylen}#1\\}

\usepackage[colorinlistoftodos,bordercolor=orange,backgroundcolor=orange!20,linecolor=orange,textsize=scriptsize]{todonotes}

\usepackage{amsmath}
\usepackage{mathtools}
\usepackage{amssymb}
\usepackage{amsthm}

\usepackage{enumerate}

\usepackage{hyperref}
\usepackage{cleveref}
\usepackage{etoolbox}
\patchcmd{\thmhead}{(#3)}{#3}{}{}

\theoremstyle{plain}
\newtheorem{theorem}{Theorem}

\newtheorem{lemma}[theorem]{Lemma} 
\newtheorem{corollary}[theorem]{Corollary} 

\theoremstyle{definition} 
\newtheorem{definition}[theorem]{Definition}
\newtheorem{example}[theorem]{Example}
\newtheorem{remark}[theorem]{Remark}

\makeatletter
\renewcommand{\todo}[2][]{\tikzexternaldisable\@todo[#1]{#2}\tikzexternalenable}
\makeatother

\newcommand{\C}{{\mathbb{C}}}

\newcommand{\R}{{\mathbb{R}}}
\newcommand{\K}{{\mathbb{K}}}
\newcommand{\F}{{\mathbb{F}}}
\newcommand{\Q}{{\mathbb{Q}}}
\newcommand{\Z}{{\mathbb{Z}}}

\newcommand{\A}{\mathcal{A}}
\newcommand{\cR}{\mathcal{R}}
\newcommand{\cD}{\mathcal{D}}
\newcommand{\cP}{\mathcal{P}}
\newcommand{\cC}{\mathcal{C}}
\newcommand{\T}{\mathcal{T}}
\newcommand{\mydef}[1]{{\color{blue}{#1}}}

\DeclareMathOperator{\ch}{ch}
\DeclareMathOperator{\codim}{codim}
\DeclareMathOperator{\Aut}{Aut}

\title[Characteristic polynomials of arrangements with symmetries]{Computing characteristic polynomials of hyperplane arrangements with symmetries} 

\author[T.\,Brysiewicz]{Taylor Brysiewicz}
\address{\textsuperscript{1}
  Department of Applied and Computational Mathematics and Statistics, University of Notre Dame}
\email{tbrysiew@nd.edu}

\author[H.\,Eble]{Holger Eble}
\address{\textsuperscript{2}
  Technische Universit\"at Berlin, Chair of Discrete Mathematics/Geometry}
\email{eble@math.tu-berlin.de}

\author[L.\,K\"uhne]{Lukas K\"uhne}
\address{\textsuperscript{3}
	Fakult\"at f\"ur Mathematik, Universit\"at Bielefeld, Bielefeld, Germany}
\email{lkuehne@math.uni-bielefeld.de}


\begin{document}

\begin{abstract}
We introduce a new algorithm computing the characteristic polynomials of hyperplane arrangements which exploits their underlying symmetry groups. Our algorithm counts the chambers of an arrangement as a byproduct of computing its characteristic polynomial. 
We showcase our \texttt{julia} implementation, based on \texttt{OSCAR}, on examples coming from hyperplane arrangements with applications to physics and computer science.
\end{abstract}
\maketitle

\section{Introduction}
The problem of enumerating chambers of hyperplane arrangements is a ubiquitous challenge in computational discrete geometry \cite{HM17,KP20,MR19,Sle99}. A well-known approach to this problem is through the computation of characteristic polynomials \cite{Ath96,HK12,KS11,OT92,ST87,Zas75}. We develop a novel for computing characteristic polynomials which takes advantage of the combinatorial symmetries of an arrangement.
While most arrangements admit few combinatorial symmetries \cite{PvP19}, most arrangements of  interest do \cite{GMP21,PS00,Zue92}.

We implemented our algorithm in \texttt{julia}~\cite{julia} and published it as the package  \texttt{CountingChambers.jl}\footnote{available at \href{https://mathrepo.mis.mpg.de/CountingChambers/index.html}{https://mathrepo.mis.mpg.de/CountingChambers}}.
Our implementation relies heavily on the cornerstones of the new computer algebra system \texttt{OSCAR} \cite{OSCAR} for group theory computations (\texttt{GAP} \cite{GAP}) and the ability to work over number fields (\texttt{Hecke} and \texttt{Nemo} \cite{nemo}). While other algorithms and pieces of software exist for studying hyperplane arrangements  (see, for instance, \cite{CS21,DP22,KP20,alcove,sage}), either their chamber-enumeration computations appear as byproducts of more difficult calculations, the code does not use symmetry, or it only pertains to very specific types of arrangements. For example, \cite{KP20} computes the associated zonotope, whose vertices are in bijection with the chambers of the arrangement, containing much more information than the characteristic polynomial. A similar approach is suggested in \cite{DP22} involving a search algorithm relying upon linear programming.
To the best of our knowledge, our implementation is the first publicly available software for counting chambers which uses symmetry.

We showcase our algorithm and its implementation on a number of well-known examples, such as the resonance and discriminantal arrangements. Additionally, we study sequences of hyperplane arrangements which come from the problem of linearly separating vertices of regular polytopes. In particular, we investigate one corresponding to the hypercube $[0,1]^d$ whose chambers are in bijection with linearly separable Boolean functions. 

In the presence of symmetry, our implementation outperforms the existing software by several orders of magnitude (cf. Table \ref{tab:CompareToOthers}). Moreover, its output is guaranteed to be correct since we compute symbolically over the integers or exact number fields and avoid overflow errors thanks to the package \texttt{SaferIntegers.jl} \cite{SaferIntegers}.


The ninth resonance arrangement ($511$ hyperplanes in $\R^{9}$) approaches the limit of what is possible with our implementation: the computation of its characteristic polynomial took $10$ days on $42$ processors. Our computation confirms that its chamber-count is $1955230985997140$ as independently and concurrently computed by Chroman and Singhar with different methods~\cite{CS21}.

We first give background on hyperplane arrangements in Section \ref{sec:Arrangements}.  The ideas outlined in Section \ref{sec:DeletionRestriction}, regarding deletion and restriction algorithms, form the basic structure of our algorithm. We explain the relevant results regarding symmetries of arrangements in Section \ref{sec:Automorphisms}. The algorithm and its implementation details reside in Section \ref{sec:OurAlgorithm}. In Section \ref{sec:Examples} we construct and discuss examples of arrangements exhibiting large symmetry groups.  We conclude in Section \ref{sec:Timings} with timings and comparisons to other software.

\section*{Acknowledgements}

We are very grateful to Tommy Hofmann, Christopher Jefferson, and Marek Kaluba for their support regarding the implementation and to Michael Cuntz for initial verifications of our computations. We would also like to thank Michael Joswig for his helpful comments throughout the project and Bernd Sturmfels for suggesting the discriminantal arrangement. Lastly, we thank the referees for their careful reading and helpful comments.

\section{Hyperplane arrangements}
\label{sec:Arrangements}

We begin by discussing background on the theory of hyperplane arrangements related to the problem of enumerating chambers: the main goal of this article and the associated software.
Our notation will mostly follow the textbook by Orlik and Terao~\cite{OT92}.

For any field $\K$, a \mydef{hyperplane} in $\K^d$ is an affine linear space of codimension one. 
Throughout this article, we denote by $\mydef{\A} = \{  H_1,\ldots,H_n\}$ a \mydef{(hyperplane) arrangement} where $H_i$ is a hyperplane in $\K^d$.

\begin{definition}
	Suppose $\A$ is an arrangement in $\R^d$.
	The connected components of the complement $\R^d\setminus \bigcup_{H\in\A}H$ are called \mydef{chambers} of $\mathcal A$ and the set of chambers of $\mathcal A$ is denoted by $\mydef{\ch(\mathcal A)}$.
\end{definition} 

\begin{example}
	\label{ex:arr}
	We use the arrangement \[\{\underbrace{\{y-x=1\}}_{H_1},\underbrace{\{x=0\}}_{H_2},\underbrace{\{x+y=1\}}_{H_3},\underbrace{\{y=0\}}_{H_4}\}\]
	in $\R^2$ as a running example.
	This arrangement is depicted in~\Cref{fig:ex}. It has $10$ chambers: $2$ bounded and $8$ unbounded.
\end{example}

\begin{figure}[!htpb]
	\begin{center}
		\includegraphics[scale=0.45]{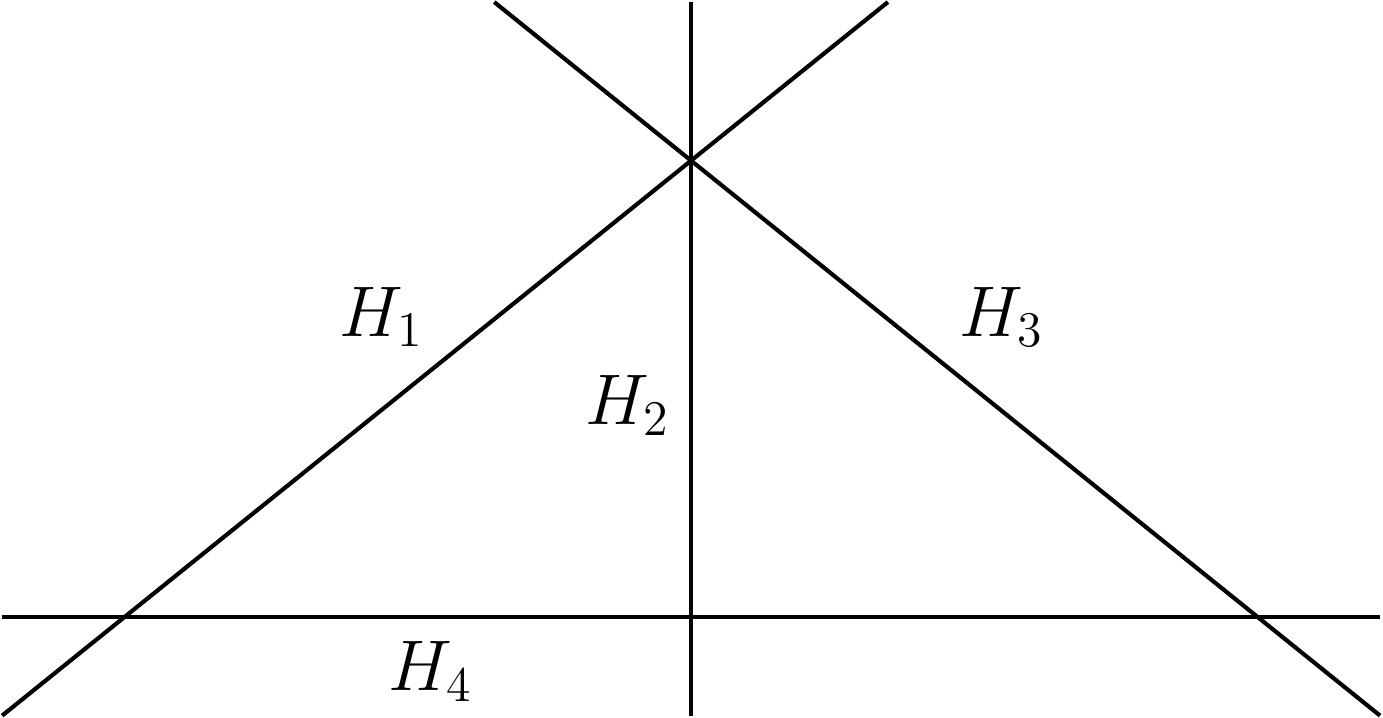}
	\end{center}
\caption{The arrangement introduced in~\Cref{ex:arr}.}
\label{fig:ex}
\end{figure}

Given a subset $I \subseteq \mydef{[n]}:=\{1,\ldots,n\}$, we write the set $\{H_i\}_{i \in I}$ as $\mydef{H_I}$ and its intersection as $\mydef{L_I}=\bigcap_{i \in I} H_i$. 
The collection of these intersections form the set $\mydef{L(\A)}= \left\{ L_I \mid I \subseteq \left[n\right], L_I \neq \emptyset \right\}$, a combinatorial shadow of $\A$ known as its \mydef{intersection poset}.
This poset is ordered by reverse inclusion and graded by the \mydef{rank function},
$\mydef{r}:L(\A)\to \Z_{\ge 0}$, where
$r(L_I)=\codim\left(L_I\right)$.
As a notational convention, we set $r(I) = r(L_I)$ for $I \subseteq [n]$ whenever $L_I \neq \emptyset$.

\subsection{The characteristic polynomial}
Our algorithm counts chambers of an arrangement by computing a more refined count, namely the characteristic polynomial. The coefficients of this polynomial are known as the unsigned Whitney numbers of the first kind of the intersection poset $L(\mathcal A)$, which we simply refer to as the \mydef{Whitney numbers} of the arrangement. 
\begin{definition}
	The \mydef{characteristic polynomial} of an arrangement $\mathcal A$ in $\K^d$ is the polynomial
	\begin{equation}\label{eq:char}
	\mydef{\chi_{\mathcal A}(t)} = \sum_{I \subseteq [n] : L_I \neq \emptyset}(-1)^{|I|} t^{d-r(I)}=\sum_{i=0}^{d} (-1)^i \mydef{b_{i}(\mathcal A)} t^{d-i}.
	\end{equation}
	The integers $b_i(\mathcal A)$, defined via~\eqref{eq:char}, are non-negative and are called the \mydef{Whitney numbers} of $\mathcal A$.
	We denote the vector of Whitney numbers by $\mydef{b(\A)}$.
\end{definition}

The characteristic polynomial and Whitney numbers of an arrangement $\A$ depend only on the intersection poset $L(\A)$ and have various interpretations depending on the field $\K$ as detailed below.

\begin{description}
	\item[Real] For an arrangement $\A$ in $\R^d$,
	Zaslavsky~\cite{Zas75} proved that 
	\[
	|\ch(\A)|=(-1)^d \chi_{\A}(-1)=\sum_{i=0}^d b_i(\A).
	\]
	Thus, the Whitney numbers are a refined count of the chambers of $\A$.
	They have the following geometric interpretation.
	Given a generic flag $\mathcal{F}_\bullet:  F_{0} \subset F_1 \subset \cdots \subset F_d=\R^d$ of affine linear subspaces $F_i$ (where $\dim(F_i)=i$)  the number of chambers of $\mathcal A$ which meet $F_i$ but do not meet $F_{i-1}$ is equal to $b_i(\A)$~\cite[Proposition 2.3.2]{Yos07}.
	
	\item[Complex] If $\A$ is an arrangement in $\C^d$ where all hyperplanes contain the origin, then $b_i(\A)$ is the $i$-th {\it{topological}} Betti number of the complement $\C^d\setminus \bigcup_{H\in\A}H$ with rational coefficients~\cite{OS80}. Because of this, some papers refer to the Whitney numbers $b_i(\mathcal A)$ as the Betti numbers of the arrangement $\mathcal A$ \cite{Yos14}.

	\item[Finite] When $\A$ is an arrangement over a finite field $\F_q$,
	Crapo and Rota proved that $\chi_{\A}(q)=|\F_q^d\setminus \bigcup_{H\in\A}H|$~\cite{CR70}.
	Moreover, if $\A$ is a hyperplane arrangement in $\Q^d$ one may consider its reduction modulo $q$: $\A\otimes\F_q=\{H_1\otimes \F_q,\dots,H_n\otimes \F_q\}$.
	When $q$ is sufficiently large, we have that $L(\A)=L(\A\otimes \F_q)$ and thus computing $\chi_{\A}(t)$ for rational arrangements also yields the number of points in the complement after reducing modulo large primes.

\end{description}

\begin{example}
Let $\A$ be the arrangement introduced in~\Cref{ex:arr}.
Its characteristic polynomial is $\chi_{\A}(t)=t^2-4t+5$.
\Cref{fig:flag} shows a generic flag~$\mathcal{F}_\bullet$ intersecting this arrangement verifying that $b(\mathcal A) = (1,4,5)$.

\vspace{1cm}

\begin{figure}[!htpb]
\includegraphics[scale=0.4]{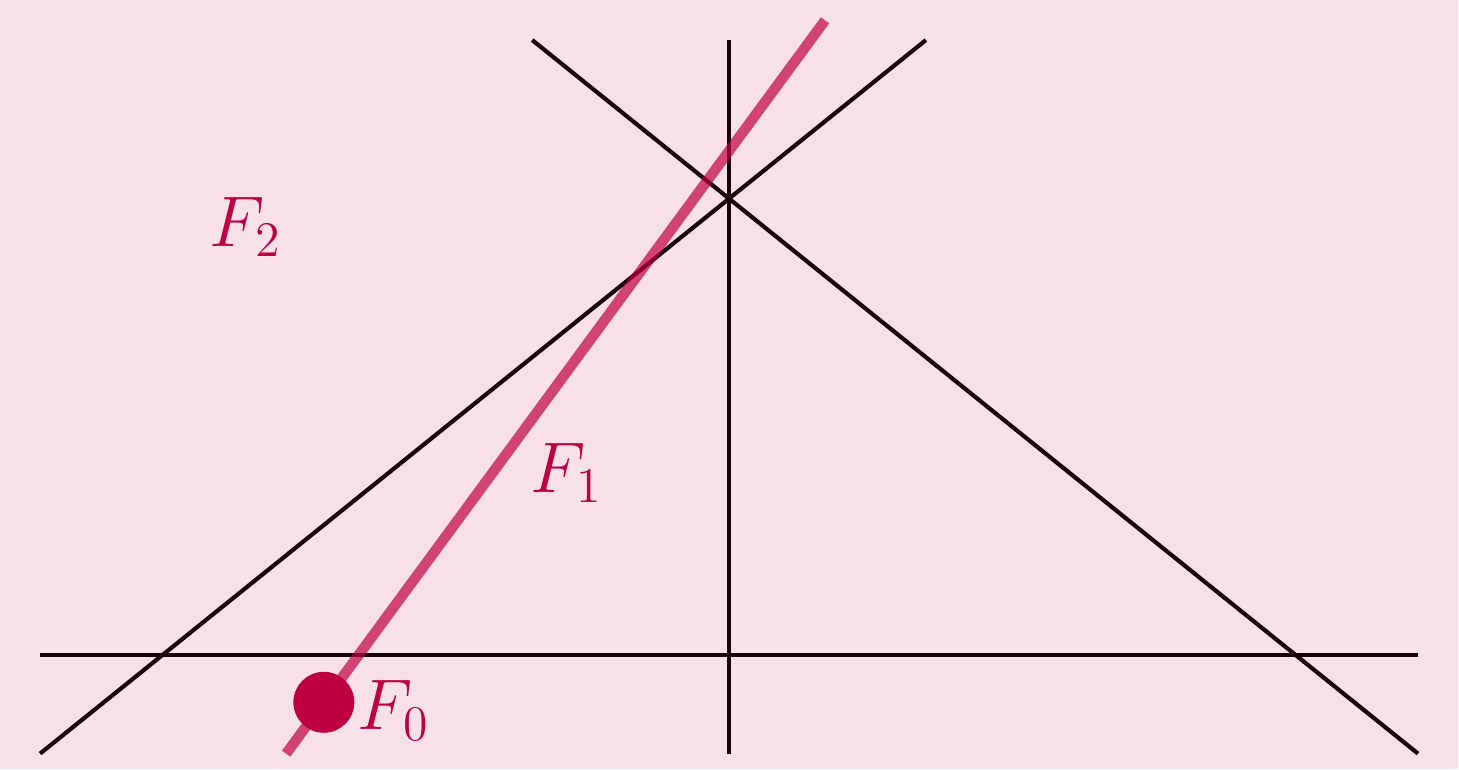}
\caption{The intersections of a generic flag (purple) in $\R^2$ with the chambers of $\mathcal A$. The point $F_0$ intersects one chamber, $F_1$ intersects four others, and $F_2$ intersects the remaining~$5$, and so $b(\mathcal A) = (1,4,5)$.}
\label{fig:flag}
\end{figure}
\end{example}

\section{A deletion-restriction algorithm}
\label{sec:DeletionRestriction}
To compute the Whitney numbers of an arrangement $\A$ in $\K^d$, we take advantage of the behavior of $\chi_{\A}(t)$ under the operations of deletion and restriction.
These operations reduce computations about $\mathcal A$ to computations about two smaller arrangements. 
Thus at its core, our main algorithm is a divide-and-conquer algorithm.

Given a hyperplane $H \in \mathcal A$, the \mydef{deletion} of $H$ in $\mathcal A$ is the arrangement ${\mathcal{A}\backslash \{H\}}$.
The \mydef{restriction} of $H$ in $\mathcal A$ is the 
arrangement in $H \cong \K^{d-1}$ defined by $\mydef{\mathcal{A}^H} = \{K \cap H \mid K \in \mathcal A \backslash \{H\}\}$.
The following lemma provides the basic foundation of our algorithm.
 \begin{lemma}[{\cite[Corollary 2.57]{OT92}}]
\label{lem:charPolyDR} 
Given a hyperplane $H \in \mathcal A$, we have that 
$\chi_{\mathcal A}(t) = \chi_{\mathcal A \backslash \{H\} }(t) - \chi_{\mathcal A^H}(t).$
In particular,
$b(\mathcal A) = b(\mathcal A\backslash \{H\}) + 0|b(\mathcal A^{H})$
where $0|b$ means prepending the vector $b$ with a zero. 
\end{lemma}
\subsection{A simple deletion-restriction algorithm}
Lemma \ref{lem:charPolyDR} along with the fact that the  empty arrangement in $\K^d$ has the vector of Whitney numbers $(1,{0,\ldots,0})\in~\mathbb{N}^{d+1}$ suggests the following well-known recursive algorithm for computing $b(\mathcal A)$.

\begin{figure}[!htpb]
\includegraphics[scale=0.22]{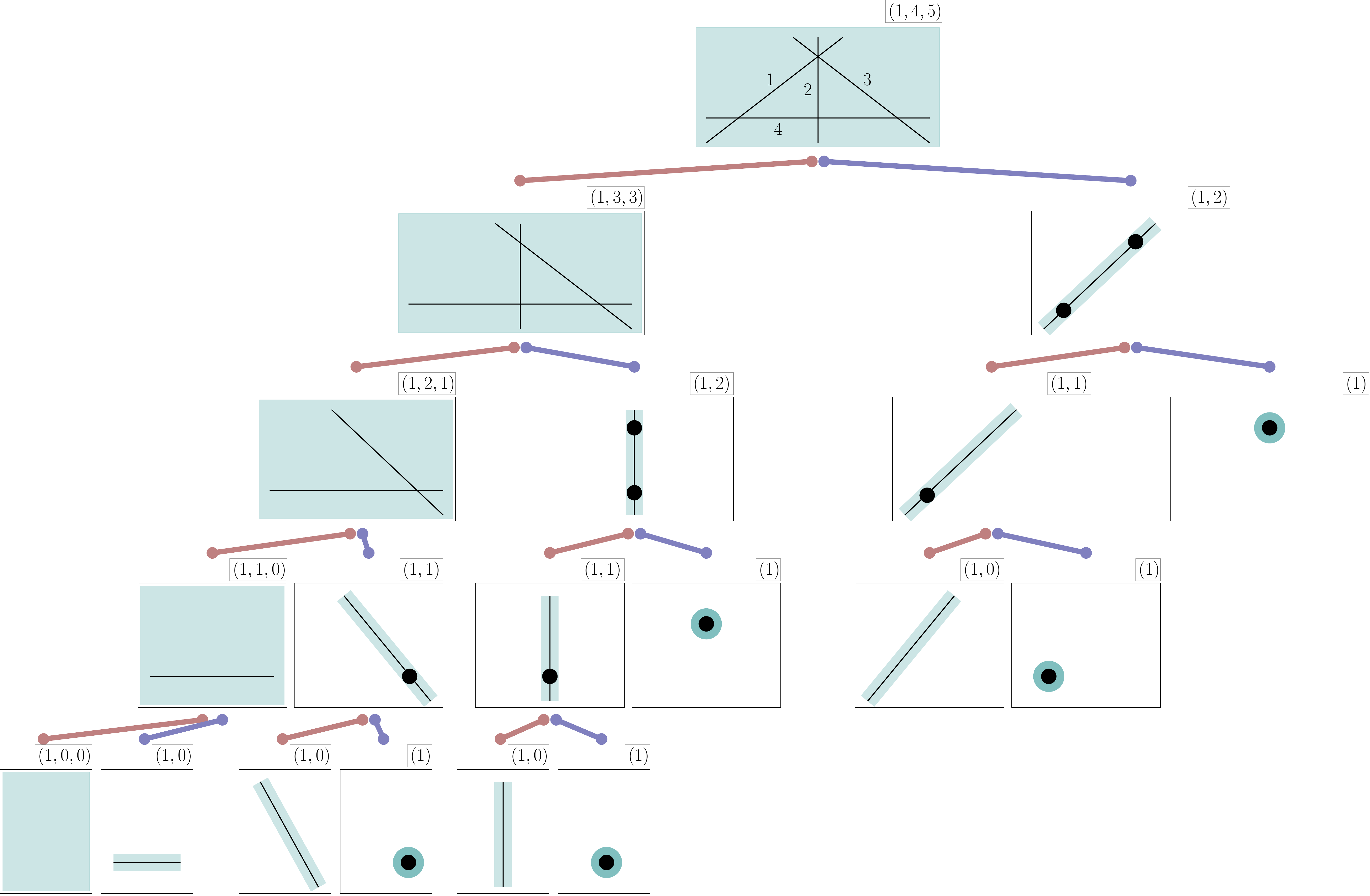}
\caption{The tree structure of Algorithm \ref{algo:simpleDR} on the hyperplane arrangement from Example \ref{ex:arr}. Hyperplanes are chosen (Line \ref{chooseHyperplane_Alg1}) according to the ordering $\{1,2,3,4\}$. In each box, the ambient space of the arrangement is shaded green. Deletions are marked with red edges (left children) and restrictions with blue edges (right children). Each arrangement box has the Whitney numbers above its upper right corner.}
\label{fig:simpletree}
\end{figure}

\begin{algorithm}[H]
	\SetKwIF{If}{ElseIf}{Else}{if}{then}{elif}{else}{}%
	\DontPrintSemicolon
	\SetKwProg{WhitneyNumbers}{WhitneyNumbers}{}{}
	\LinesNotNumbered
	\KwIn{A hyperplane arrangement $\mathcal A$ in $\K^d$
	}
	\KwOut{The vector of Whitney numbers $b(\mathcal A)$}
	\WhitneyNumbers(){$(\mathcal A)$}{
		\nl  \uIf{$\emptyset \neq \mathcal A$}{
			\nl\label{chooseHyperplane_Alg1}choose $H \in \mathcal A$ \;
			\nl \Return{{\small \bf{WhitneyNumbers}($\mathcal A \backslash \{H\}$)~+~$0|$\textbf{WhitneyNumbers}~($\mathcal A^H$)}}\;
		} 
 \nl		\Else{
	\nl		\Return{$(1,0,\ldots,0)$}
		}
	}
	\caption{Whitney numbers via simple deletion and restriction \label{algo:simpleDR}}
\end{algorithm}

\vspace{1cm}

Structurally, Algorithm \ref{algo:simpleDR} is a depth-first binary tree algorithm on arrangements, rooted at the initial input: one child represents a deletion and the other a restriction, as shown in Figure \ref{fig:simpletree}.

The implementation of Algorithm \ref{algo:simpleDR} is already nontrivial as it is often the case that some hyperplanes become the same after a restriction.  Thus, its proper implementation requires care in representing an arrangement on a computer.

\subsection{Computationally representing deletions and restrictions}\label{ComputerRepresentation}
An arrangement $\mathcal B$ coming from $\mathcal A$ via deletions and restrictions may be represented by an encoding of the restricted hyperplanes. To be precise, the~pair \[\mydef{B} = (\{H_{i_1},\ldots,H_{i_k}\},\{H_{j_1},\ldots,H_{j_\ell}\})=:(H_I,H_J)\]   represents the hyperplane arrangement $\mathcal B$ in $L_I \cong \K^{d-r(L_I)}$ given by the {\it hyperplanes} in $\{H_j \cap L_I\}_{j \in J}$. Note that $H_j \cap L_I$ may be empty for some~$j \in J$, in which case this intersection does not correspond to any hyperplane. We extend notation regarding $\mathcal B$ to its representation $B$ (i.e.~$\chi_B(t) := \chi_{\mathcal B}(t)$ and $b(B) := b(\mathcal B)$).

If $H_{j_1} \cap L_I$ is a hyperplane which occurs uniquely with respect to the tuple $(H_{j_1} \cap L_I,\ldots,H_{j_\ell} \cap L_I)$, then $\mathcal B^{H_{j_1} \cap L_I}$ 
and $\mathcal B\backslash \{H_{j_1} \cap L_I\}$ are represented by 
\begin{align*}
\mydef{B^{H_{j_1}}} &:=  (\{H_{i_1},\ldots,H_{i_k},H_{j_1}\},\{H_{j_2},\ldots,H_{j_\ell}\})\\
\mydef{B \backslash \{H_{j_1}\}} &:=  (\{H_{i_1},\ldots,H_{i_k}\},\{H_{j_2},\ldots,H_{j_\ell}\}), \end{align*}
respectively.
Whereas if $H_{j_1} \cap L_I$ is either empty or does not occur uniquely, then $B \backslash{\{H_{j_1}\}}$ trivially represents the same arrangement as $B$, namely $\mathcal B$. 

The following computational analogue of Lemma \ref{lem:charPolyDR} establishes how such representations behave under deletion and restriction. 
\begin{lemma}\label{lem:charPolyDRpairs}
Let $B = (H_I,H_J)$ represent an arrangement $\mathcal B$ and fix $H \in H_J$. If $H \cap L_I$ is a hyperplane which occurs uniquely in the tuple $(H_j \cap L_I)_{j \in J}$ then 
$\chi_{B}(t) = \chi_{ B \backslash{\{H\}}}(t)-\chi_{ B^H}(t)$ and $b(B) = b(B\backslash \{H\}) + 0|b(B^{H})$.
Otherwise, we have
$\chi_{ B}(t) = \chi_{ B \backslash{\{H\}}}(t)$ and $b(B) = b(B\backslash\{H\}).$

\end{lemma}
\begin{proof}
The first case follows from Lemma \ref{lem:charPolyDR}. In the second case, $B$ and~$B\backslash \{H\}$ represent the same hyperplane arrangement and the result is trivial. 
\end{proof}
The following algorithm is equivalent to Algorithm \ref{algo:simpleDR}.

\begin{algorithm}[H]
	\SetKwIF{If}{ElseIf}{Else}{if}{then}{elif}{else}{}%
	\DontPrintSemicolon
	\SetKwProg{WhitneyNumbers}{WhitneyNumbers}{}{}
	\LinesNotNumbered
	\KwIn{A representation $B=(H_I,H_J)$ of an arrangement in $\K^d$
	}
	\KwOut{The vector of Whitney numbers $b(B)$}
	\WhitneyNumbers(){$B=(H_I,H_J)$}{
		\nl  \uIf{$\emptyset \neq H_J$}{
			\nl\label{chooseHyperplane_Alg2}choose $H \in H_J$ \;
			\nl\label{uniquenessCond}\uIf{$H\cap L_I\neq \emptyset$ occurs uniquely in $(H_{j} \cap L_I)_{j \in J}$}{
			\nl \Return{{\small \bf{WhitneyNumbers}($B\backslash~\{H\}$)~+~$0|$\bf{WhitneyNumbers}~(~$B^H$)}}\;
			}
\nl			\Else{
\nl			\Return{{\small \bf{WhitneyNumbers}($B\backslash \{H\}$)}}
}
		} 
 \nl		\Else{
	\nl		\Return{$(1,0,\ldots,0)$}
		}
	}
	\caption{Whitney numbers via extended deletion and restriction \label{algo:extendedDR}}
\end{algorithm}

Given a hyperplane arrangement $\mathcal A = \{H_1,\ldots,H_n\}$ in $\K^d$, Algorithm~\ref{algo:extendedDR} computes the Whitney numbers $b_i(\mathcal A)$ when given $A = (\emptyset,\{H_1,\ldots,H_n\})$ as input. This algorithm traverses a binary tree which is essentially the same as the one from Algorithm \ref{algo:simpleDR}. The only difference is that some edges are extended with nodes that have only one child and so we say it computes the Whitney numbers via {\it{extended}} deletion and restriction.

Algorithm \ref{algo:extendedDR} has the advantage that the representations of the original hyperplanes in $\mathcal A$ need not be updated upon restriction, and that representations of hyperplanes in $\mathcal A^H$ need not be unique. As a consequence, structural aspects of $\mathcal A$ such as its symmetries extend trivially to the representations of the restricted arrangements, as we explain in Section \ref{sec:Automorphisms}.
Figure~\ref{fig:DRTreeExtended} displays the tree structure underlying Algorithm \ref{algo:extendedDR} on our running example. Note that $J$ is constant amongst nodes in the same depth.

\begin{figure}[!htpb]
\includegraphics[scale=0.25]{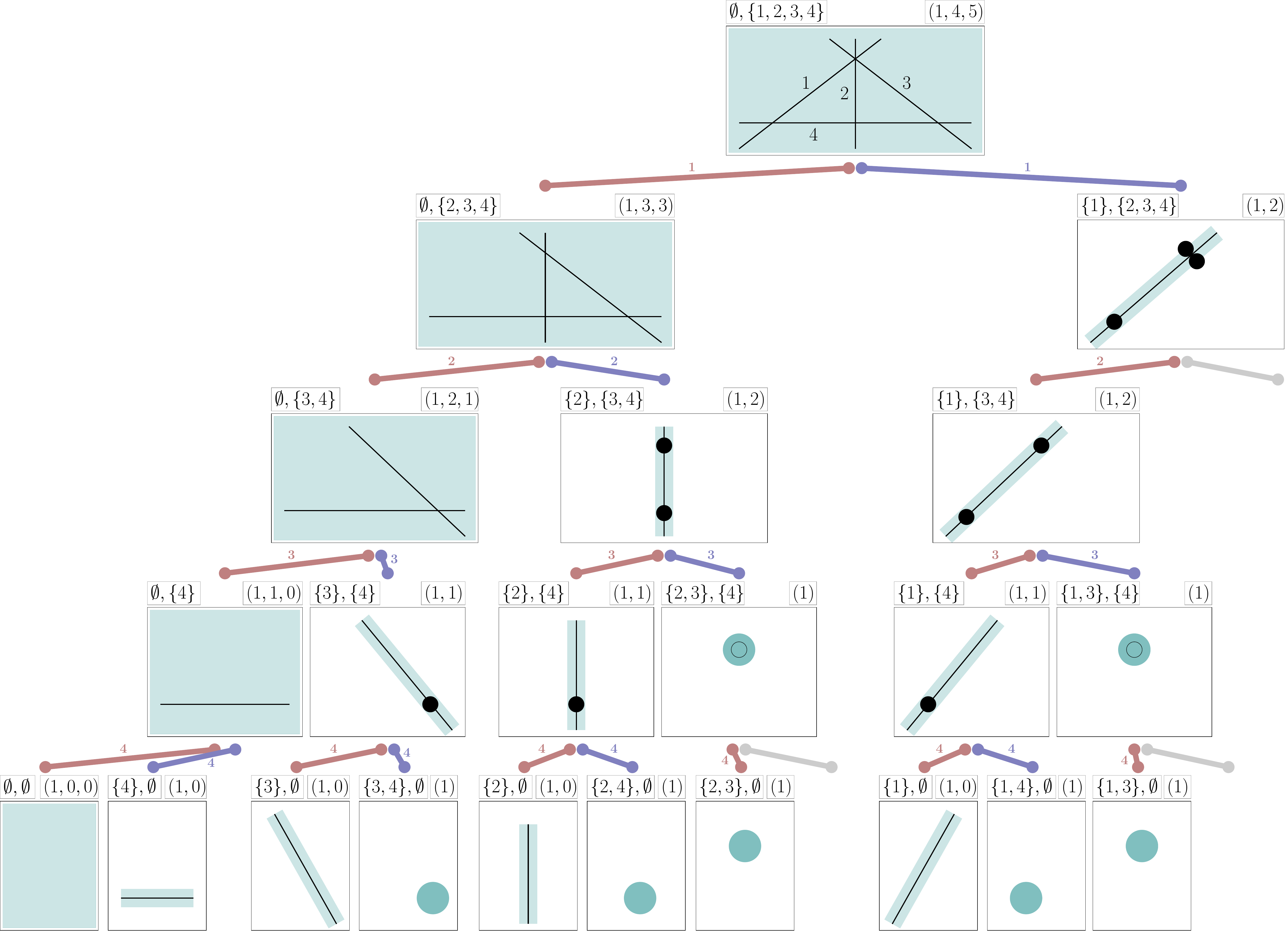}
\caption{The tree structure of Algorithm \ref{algo:extendedDR} on the hyperplane arrangement from Example \ref{ex:arr}. Its nodes are represented by pairs of subsets $I,J \subset \{1,2,3,4\}$ (top-left) and the Whitney numbers are given (top-right). Grey edges indicate that the condition in Line \ref{uniquenessCond} has been violated.
}
\label{fig:DRTreeExtended}
\end{figure}

\section{Automorphisms of hyperplane arrangements}
\label{sec:Automorphisms}
Our main contribution is the inclusion of symmetry-reduction in the deletion-restriction algorithm. Many other algorithms in discrete geometry have also been adapted to take advantage of symmetry \cite{Bremner1,Bremner2,Jensen,JJK}. For us, the relevant symmetries for an arrangement are the rank-preserving permutations of its hyperplanes.

Let $\mydef{\mathfrak S_n}$ be the permutation group on $[n]$. Elements of a subgroup $G \leq \mathfrak S_n$   act on subsets of $[n]$. Given $g \in G$ and $I \subseteq [n]$, we fix the notation \vspace{-0.2em}
\begin{itemize}
\item[-] $\mydef{gI}=\{g(i)\}_{i \in I}$ for the \mydef{image} of $I$ under $g$,
\item[-] $\mydef{I^G}=\{g \in G \mid gI=I\}$ for the \mydef{stabilizer} of $I$ in $G$,
\item[-] $\mydef{G\cdot I}=\{gI \mid g \in G\}$ for the \mydef{orbit} of $I$ under $G$.
\end{itemize}
\begin{definition} The \mydef{automorphism group} of $\mathcal A=\{H_1,\ldots,H_n\}$ is
$$\mydef{\Aut(\mathcal A)} = \{g \in \mathfrak S_n \mid r(H_I) = r(H_{gI}) \text{ for all } I \subseteq [n]\}.$$
\end{definition}

Given a representation $B=(H_I,H_J)$ of an arrangement coming from $\A$, the automorphism group $\Aut(\A)$ acts as $\mydef{gB}=(H_{gI},H_{gJ})$.

\begin{remark}
Our definition of the automorphism group of an arrangement is combinatorial, not geometric. This difference can be quite large. For example, a generic hyperplane arrangement $\mathcal A=\{H_1,\ldots,H_n\}$ has no geometric symmetries but $\textrm{Aut}(\mathcal A) = \mathfrak S_n$.
\end{remark}

\begin{lemma}
\label{lem:automorphisms}
Let $\mathcal A=\{H_1,\ldots,H_n\}$ be an arrangement in $\K^{d}$ and let $B_1$ and $B_2$ represent arrangements coming from deletions and restrictions. If $B_1$ and  $B_2$ are in the same orbit under $\Aut(\A)$ then $b(B_1) = b(B_2)$.
\end{lemma}
\begin{proof}
The conclusion of the lemma is equivalent to showing that the characteristic polynomials of $B_1$ and $B_2$ are the same. This follows directly from the fact that the characteristic polynomial depends only on the intersection poset (graded by rank) and that $B_1$ and $B_2$ are in the same orbit under $\Aut(\A)$ if and only if they are related by a rank-preserving permutation.
\end{proof}

Our algorithm relies upon the following corollary of Lemma \ref{lem:automorphisms}.
\begin{corollary}
\label{cor:stabAutomorphism}
Let $B=(H_I,H_J)$ represent a hyperplane arrangement coming from $\mathcal A=\{H_1,\ldots,H_n\}$. For $g \in {J}^{\Aut(\mathcal A)}$ we have that $gB=(H_{gI},H_J)$ and $B$ have the same Whitney numbers.
\end{corollary}

\section{Enumeration algorithm with symmetry}
\label{sec:OurAlgorithm}
Our main algorithm augments Algorithm \ref{algo:extendedDR}, making particular use of Corollary \ref{cor:stabAutomorphism}. It is essentially a breadth-first tree algorithm except that at each level, nodes may be identified up to symmetry and so the algorithmic structure is no longer that of a tree. The output is the vector of Whitney numbers $b(\mathcal{A})$ of an arrangement $\mathcal A$, refining its chamber count. We remark that despite the fact that our algorithm takes advantage of symmetry and counts the number of chambers, it does not reveal any information about the sizes of orbits of chambers under this symmetry group.

Given an arrangement $\A=\{H_1,\ldots,H_n\}$ in $\K^d$, we represent the nodes of the algorithm at depth $k$ by a dictionary $T_k$. The keys of $T_k$ are orbits~$G_k\cdot I$ for $I \subseteq [k]$ where $G_k$ is a subgroup of the stabilizer of $\{k+1,\ldots,n\}$ in $\Aut(\A)$. The value of $G_k\cdot I$ in this dictionary is a pair $(B_I,\omega(B_I))$ where $B_I$ represents the hyperplane arrangement $(H_I,H_{\{k+1,\ldots,n\}})$ and $\mydef{\omega(B_I)}$ is some multiplicity, tracking how many arrangements indexed by elements of the orbit $G_k\cdot I$ have appeared. We refer to $T_k$ as a \mydef{$k$-th orbit-node dictionary}.

Algorithm \ref{algo:DRwSymmetry} presents the breadth-first structure of the algorithm.

\begin{algorithm}[H]
	\SetKwIF{If}{ElseIf}{Else}{if}{then}{elif}{else}{}%
	\DontPrintSemicolon
	\SetKwProg{WhitneyNumbers}{WhitneyNumbers}{}{}
	\LinesNotNumbered
	\KwIn{A hyperplane arrangement $\mathcal A=\{H_1,\ldots,H_n\}$ in $\K^d$ 
	}
	\myinput{A subgroup $G \leq \Aut(\A)$}
	\KwOut{The vector of Whitney numbers $b(\mathcal A)$}
	\WhitneyNumbers(){$(\mathcal A)$}{
		\tcp{compute the stabilizers of $G$}
		\nl compute 	
	$\{G_i\}_{i=0}^n$ where $G_i = \{i+1,\ldots,n\}^{G}$ and $G_n=G$\;
		\tcp{initialize orbit-node dictionaries} 
		\nl initialize $\{T_i\}_{i=0}^n$ and set $T_0=\{G_0\cdot \emptyset\Rightarrow ((\emptyset,\A),1)\}$\;
		\nl  \For{$k=1,\ldots,n$}{
			\nl set $T_k = \textbf{NextGeneration}(\mathcal A,G_k,T_{k-1})$
		}  
		\nl initialize $b=(0,0,\ldots,0)$\;
	\nl \For{$(B_I,\omega(B_I)) \in T_n$}{
	\nl increment the entry $b_{|I|}$ by $\omega(B_I)$\;
	}
	\nl \Return{b}
	}
	\caption{Whitney numbers using symmetry \label{algo:DRwSymmetry}}
\end{algorithm}

 Moving from depth $k-1$ to $k$ is performed by Algorithm \ref{algo:sproutleaves}. 

\begin{algorithm}[H]
	\SetKwIF{If}{ElseIf}{Else}{if}{then}{elif}{else}{}%
	\DontPrintSemicolon
	\SetKwProg{NextGeneration}{NextGeneration}{}{}
	\LinesNotNumbered
	\KwIn{A hyperplane arrangement $\mathcal A=\{H_1,\ldots,H_n\}$ in $\K^d$
	}
	\myinput{A subgroup $G_k \leq \{k+1,\ldots,n\}^{\Aut(\A)}$}
	\myinput{An orbit-node dictionary $T_{k-1}$
	}
	\KwOut{An orbit-node dictionary $T_k$}
	\NextGeneration(){$(\mathcal A,G_k,T_{k-1})$}{ \nl set $J=\{k+1,\ldots,n\}$ \;
		\nl  \For{$(B_I,\omega(B_I)) \in \texttt{values}(T_{k-1})$}{
			\nl\label{line:ConditionSymm}\If{$H_{k}\cap L_I$ is a unique hyperplane amongst $(H_{j} \cap L_I)_{j=k}^n$}{
			\tcp{produce the restriction as the right child}
			\nl set $I'=I \cup \{k\}$\;
			\nl\label{line:orbit1}compute the orbit $\mathcal O=G_{k}\cdot I'$\;
			\nl \uIf{$\mathcal O \in \texttt{keys}(T_k)$}{
			\nl increment the multiplicity of $T_k[\mathcal O]$ by $\omega(B_I)$\;}
			\nl \Else{
		\nl	$T_k[\mathcal O]=((H_{I'},H_{J}),\omega(B_I))$\;
			}
			}
			\tcp{produce the deletion as the left child}
		\nl\label{line:orbit2}compute the orbit $\mathcal O = G_k\cdot I$\;
			\nl \uIf{$\mathcal O \in \texttt{keys}(T_k)$}{
			\nl increment the multiplicity of $T_k(\mathcal O)$ by $\omega(B_I)$\;}
			\nl \Else{
		\nl	$T_k[\mathcal O]=((H_I,H_J),\omega(B_I))$\;
			
}
		} 
\Return{$T_{k}$}\;
	}
	\caption{NextGeneration \label{algo:sproutleaves}}
\end{algorithm}

\begin{example}
The structure underlying Algorithm \ref{algo:DRwSymmetry} applied to the arrangement in Example \ref{ex:arr} is shown in Figure \ref{fig:SymmTree}. It is no longer a tree but may be obtained from the tree in Figure \ref{fig:DRTreeExtended} by identifying nodes under the stabilizers of $\Aut(\A)$. Each identification accumulates multiplicity in the node and that multiplicity is passed down to its children.
\begin{figure}[!htpb]
\includegraphics[scale=0.25]{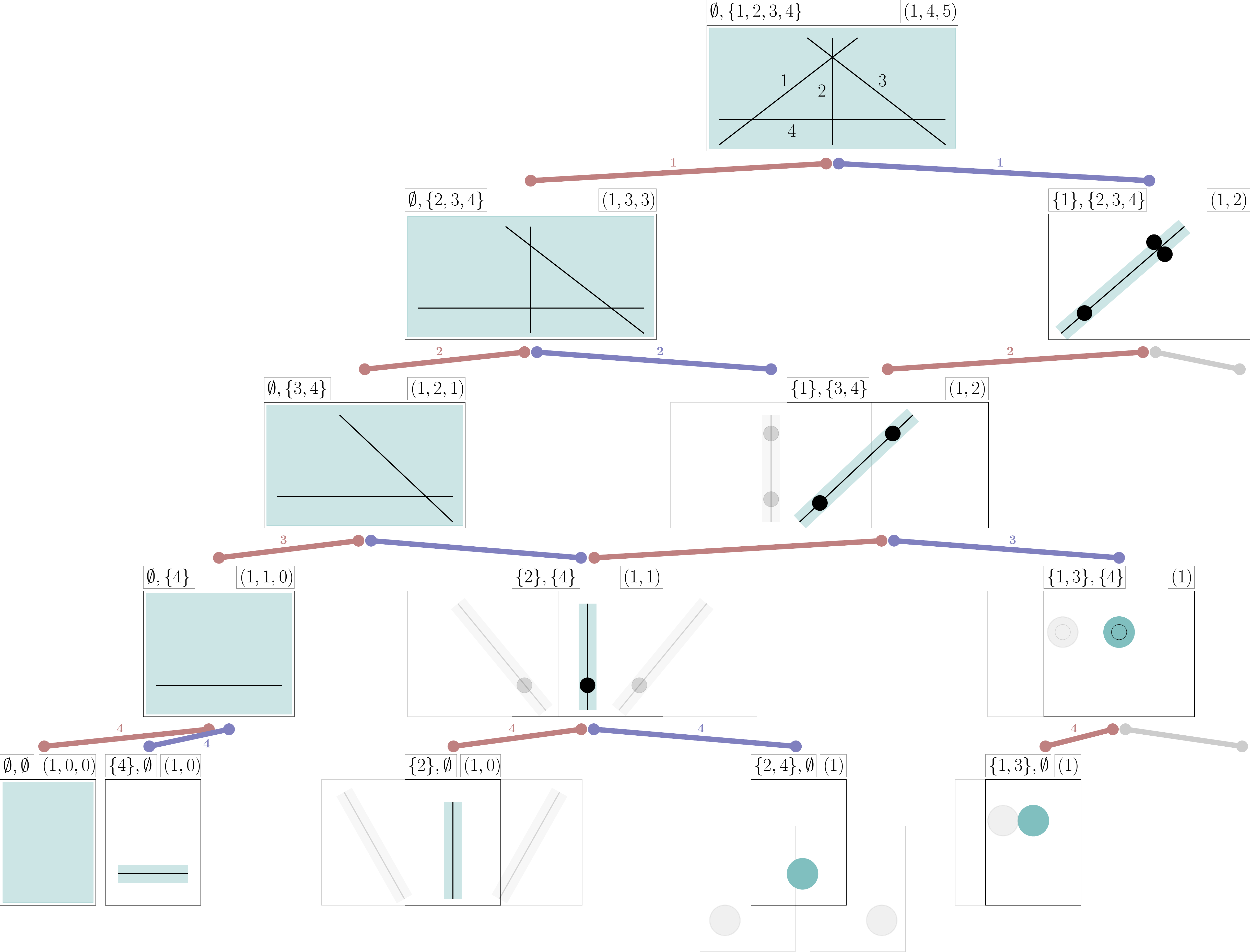}
\caption{The algorithmic structure underlying Algorithm \ref{algo:DRwSymmetry}. Starting at the top node, each call of Algorithm \ref{algo:sproutleaves} produces the next depth of this graph.}
\label{fig:SymmTree}
\end{figure}

\end{example}

\subsection{Representing orbits}
The computations of orbits in Line \ref{line:orbit1} and Line \ref{line:orbit2} require elaboration; specifically in regards to representing an orbit $G\cdot I$ on a computer. One option is to use a canonical element of $G\cdot I$, which can be computed using the \texttt{MinimalImage} or \texttt{CanonicalImage} functions from \texttt{GAP}~\cite{JJPW19,images}. An alternative approach is to provide any function $\varphi: 2^{[n]} \to S$ taking values in an arbitrary set $S$ such that $\varphi(I)=\varphi(J)$ only if $G\cdot I=G\cdot J$. Equivalently, $\varphi$ is any factor of the projection $\pi:2^{[n]} \to  2^{[n]}/G$ as a map of sets where $2^{[n]}/G$ is the set of orbits. In this case, the value of $\varphi(I)$ may be used to represent the orbit $G\cdot I$ as a key in the orbit-node dictionaries. While this approach may fail to identify all nodes in the same orbit, nodes in distinct orbits are never identified and so the algorithm remains correct. The benefit is that it may be significantly more efficient to evaluate $\varphi$ than it is to compute minimal or canonical images. 

Our default option for identifying orbits is called \texttt{pseudo\_{\hspace{0.07em}}minimal\_image}. Given a subset $I\subseteq [n]$ and a collection of elements $g_1,\ldots,g_m~\in~G~\leq~\mathfrak{S_n}$, this function sequentially computes $g_iI$ and recursively calls itself on $g_iI$ whenever $g_iI<I$ lexicographically. If no such $g_i$ produces a smaller subset, $I$ itself is returned. Options are implemented for choosing $m$ to be a proportion of $|G|$ subject to maximum and minimum values. For our computations, we take $m=n$ random elements of $G$. Although this greedy procedure does not make all possible identifications in the algorithm, we have found that it is quicker than \texttt{MinimalImage} to evaluate and produces a comparably small algorithmic structure. 

\begin{example}
We compare the effect of three choices of identifications in Algorithm \ref{algo:DRwSymmetry} (either \texttt{pseudo\_{\hspace{0.07em}}minimal\_image}, the \texttt{MinimalImage} function in \texttt{GAP}, or no identifications at all) on the resonance arrangement $\mathcal{R}_7$ (see Definition \ref{def:resonance}) consisting of $127$ hyperplanes in $\mathbb{R}^7$. We compare the number of leaves of the algorithm at some depth, as well as the time per depth of the algorithm and display the results in Figure \ref{fig:LeavesAndTimings}. 

As depicted, the cost (in number of leaves) of using \texttt{pseudo\_{\hspace{0.07em}}minimal\_image} compared to \texttt{MinimalImage} is negligible, while the benefits in terms of speed are significant. Similarly, while the timing of our algorithm with \texttt{MinimalImage} is comparable to the timing without any identifications (Algorithm \ref{algo:extendedDR}), the memory usage is significantly reduced as conveyed by the number of leaves (a reasonable proxy for memory usage). This difference becomes even more dramatic for larger arrangements.

\begin{figure}[!htpb]
	\begin{subfigure}[b]{0.51\linewidth}
		\centering
		\includegraphics[scale=.3]{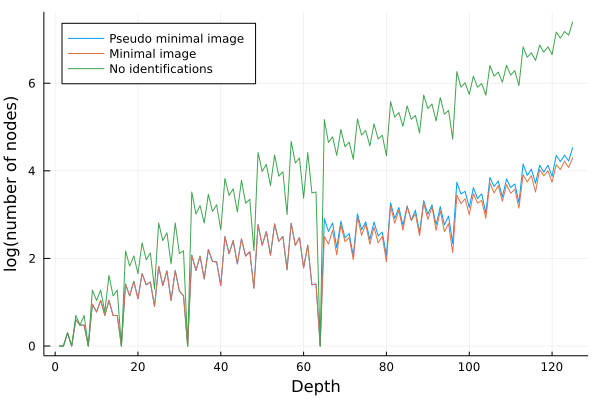}
		\caption{Leaves at depth.}
	\end{subfigure}
	\begin{subfigure}[b]{0.51\linewidth}
		\centering
		\includegraphics[scale=.3]{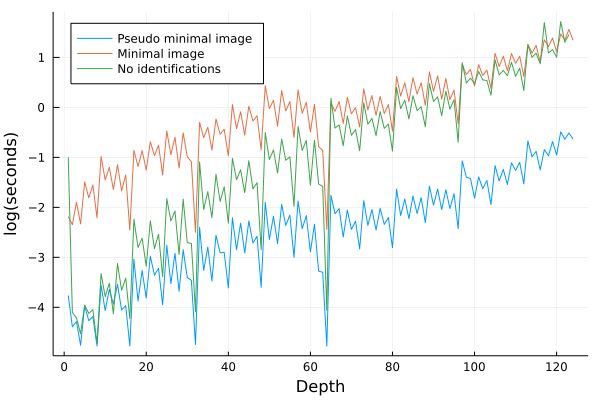}
		\caption{Time per depth.}
	\end{subfigure}
	\caption{The leaves per depth and time per depth of Algorithm \ref{algo:DRwSymmetry} on the arrangement $\mathcal{R}_7$ using \texttt{pseudo\_{\hspace{0.07em}}minimal\_image}, \texttt{MinimalImage}, and no identifications.}
\label{fig:LeavesAndTimings}
\end{figure}

\end{example}

\subsection{Accumulating the Whitney numbers and skipping levels}
Much of the computational burden occurs in Line \ref{line:ConditionSymm} of Algorithm \ref{algo:sproutleaves} and involves projecting the normal vectors of the hyperplanes in $\A$ along those hyperplanes which have been restricted. When implementing Algorithm \ref{algo:sproutleaves}, one may choose whether to save such computations at the cost of memory, or to perform redundant computations throughout the algorithm. We found that, for our benchmark examples, recomputation held the most benefit.

Nonetheless, from the linear algebra involved in the evaluation of Line \ref{line:ConditionSymm}, one can read off $\mydef{j_{\min}}$, the smallest $j \in J$ for which this uniqueness condition is true. Hence, one may immediately place the left child of the corresponding node in level $j_{\min}$ rather than $k$ to avoid redundancy in Line \ref{uniquenessCond} later on. This comes at the cost of missing some identifications between the layers $k$ and $j_{\min}$. 

Another implementation choice we made was to keep a running count of the Whitney numbers of the arrangement throughout the algorithm. Whenever $j_{\min}=n$ while computing the children of $(B_I,\omega(B_I))$, we increment $b_{|I|}$ by $\omega(B_I)$ and delete the node altogether since no other deletions or restrictions are possible. Similarly, if $\A$ is a hyperplane arrangement where each hyperplane contains the origin, $b_{|I|}$ and $b_{|I|+1}$ are incremented by $\omega(B_I)$ whenever $j_{\min}=n-1$ by a similar reasoning. In this way, we can free memory occupied by nodes throughout the algorithm. 

\subsection{Relation to \texttt{OSCAR}}

The new computer algebra system \texttt{OSCAR} in \texttt{julia} combines the existing systems \texttt{GAP} \cite{GAP}, \texttt{Singular} \cite{Singular}, \texttt{Polymake} \cite{Polymake,Polymakejl}, and \texttt{Antic~(Hecke, Nemo)}~\cite{OSCAR}.
Our software is written in \texttt{julia} and builds heavily on these cornerstones.
Specifically, we use the number theory components \texttt{Nemo}~\cite{nemo} and \texttt{Hecke} to work with arrangements defined over algebraic field extensions of $\Q$.
For example the separability arrangement of the vertices of the $600$-cell is defined over $\Q(\sqrt{5})$.

Secondly, we use \texttt{GAP}~\cite{GAP} for group theoretic computations in Algorithm~\ref{algo:sproutleaves}.
Concretely, we compute stabilizers and minimal images using the \texttt{GAP} packages \texttt{ferret}~\cite{ferret} and \texttt{images}~\cite{images}, respectively.

\subsection{Functionality of \texttt{CountingChambers.jl}}

The \texttt{julia} package titled \texttt{CountingChambers.jl} contains our implementation and is available at 
\begin{center}
 \href{https://mathrepo.mis.mpg.de/CountingChambers/index.html}{https://mathrepo.mis.mpg.de/CountingChambers}
 \end{center}
The following code snippet shows some standard functions of our package applied to the arrangement introduced in~\Cref{ex:arr}.
A collection of hyperplanes defined by the equations $\ell_i(x_1,\dots,x_d)=c_i$ for $1\le i \le n$ is encoded by a $d\times n$ matrix $A$ having the coefficients of $\ell_i$ as columns and a vector~$c$.

 {\scriptsize
\begin{verbatim}
           julia> A = [-1 1 1 0; 1 0 1 1];
           julia> c = [1, 0, 1, 0];
           julia> whitney_numbers(A; ConstantTerms=c)
           3-element Vector{Int64}:
           1 4 5
           julia> characteristic_polynomial(A; ConstantTerms=c)
           t^2 - 4*t + 5
           julia> number_of_chambers(A; ConstantTerms=c)
           10
\end{verbatim}
}
Note that the automorphism group of this arrangement is $\mathfrak{S}_3\hookrightarrow \mathfrak{S_4}$ consisting of permutations of the first three hyperplanes.
This group can be passed to our algorithm via a list of generators in one-line notation:
{\scriptsize
\begin{verbatim}
           julia> G = [[2,3,1,4],[2,1,3,4]];
           julia> whitney_numbers(A; ConstantTerms=c, SymmetryGroup=G)
           3-element Vector{Int64}:
           1 4 5
\end{verbatim}
}

As it is easy to run \texttt{julia} on multiple threads, we also implemented our algorithm to take advantage of this.
By starting \texttt{julia} via the command \texttt{julia --threads NUM\_THREADS} and passing the optional parameter \texttt{multi\_threaded=true} to our methods, the \textbf{for} loop in Algorithm \ref{algo:sproutleaves} is executed in parallel.  Table \ref{tab:parallel} shows how the multithreading scales.

\section{Examples and integer sequences}
\label{sec:Examples}

We apply our algorithm to a number of examples. Many of these arise from the following construction of separability arrangements.

\subsection{Separability arrangements}\label{sec:separability}
Fix a finite set $\mathbf{V} \subset \R^d$.
 We associate to every $v \in {\mathbf{V}}$ the hyperplane $\mydef{H_v} \subset (\R^{d+1})^*$ comprised of linear forms which vanish on $(1,v)$. Equivalently, $H_v$ represents the affine hyperplanes in $\R^d$ which contain $v$. We call the arrangement 
$\mydef{\mathcal H_{\mathbf{V}}} :=\{H_v \mid v \in \mathbf{V}\}$
the \mydef{separability arrangement} of $\mathbf V$. We point out that by increasing the dimension $d$ by one, this construction is distinct from the one which defines \emph{real reflection arrangements} from root systems. In particular, translating $\mathbf{V}$ does not change the combinatorics of $\mathcal H_{\mathbf{V}}$.

A hyperplane $H_v$ partitions the points in $(\R^{d+1})^*\backslash H_v$ into the sets $\mydef{H^+_{v}}$ of linear forms which are positive on $v$ and $\mydef{H^-_v}$ which are negative on $v$. Consequently, all affine hyperplanes corresponding to points in a chamber of $\mathcal H_{\mathbf{V}}$ are positive on some subset $V_1 \subset \mathbf{V}$ and negative on its complement $V_2=\mathbf{V}\backslash V_1$. Such a  partition $V_1 \sqcup V_2 = \mathbf{V}$ is called \mydef{linearly separable}. Hence, chambers of $\mathcal H_{\mathbf{V}}$ are in bijection with linearly separable partitions of $\mathbf{V}$, motivating the terminology for $\mathcal H_{\mathbf{V}}$. This point of view, which connects linear separability and hyperplane arrangements, appears in \cite[Section 2]{Baldi}.

One purpose for introducing separability arrangements is that it immediately provides us with a zoo of arrangements admitting considerable symmetry; for example, those $\mathbf{V}$ which are the vertices of regular polytopes.

\subsection{The threshold arrangement}
The following arrangement appears in the study of neural networks \cite{Montufar2,Montufar1,Zunic} and algebraic statistics \cite{Cueto}. 
\begin{definition}\label{def:threshold}
The \mydef{threshold arrangement}\footnote{The arrangement $\{x_i+x_j\}_{1 \le i < j \le d}$ in $\R^d$ is also referred to as a threshold arrangement in the literature. We discuss the arrangement $\mathcal T_d$ only as in Definition~\ref{def:threshold}.}, $\mydef{\mathcal T_d}$ is the separability arrangement associated to the vertices of the hypercube $[0,1]^{d}$. That is,\vspace{-0.3em}
\[
\mydef{\T_d} \coloneqq \{ \{x_0 + c_1x_1 + \dots + c_dx_d=0 \} \mbox{ with } c_i\in \{0,1\} \mbox{ for all }c_i\}.
\]
\end{definition}\vspace{-0.3em}
As a consequence of the definition of $\mathcal T_d$, the linear automorphisms of the hypercube $[0,1]^d$, namely the hyperoctahedral group of order $d!2^d$, is a subgroup of $\Aut(\mathcal T_d)$. The true size of $\Aut(\mathcal T_d)$ is $(d+1)!2^d$. 

We computed the Whitney numbers of $\T_d$ for $1\le d \le 8$, and thus their number of chambers. The results are collected in Table~\ref{tab:tn} and the timings appear in Table \ref{tab:OurTimings}.
The values of $|\ch(\T_d)|$ for $1\le d \le 9$ are listed in entry \href{https://oeis.org/A000609}{A000609} of the Online-Encyclopedia of Integer Sequences (OEIS), whereas the Whitney numbers of $\T_d$, to the best of our knowledge, have not been published before.
Zuev showed that asymptotically $ |\ch(\T_d)| \sim 2^{d^2}$~\cite{Zue92}. 

\begin{remark}
	Using similar proof techniques as in~\cite{Kue20} one can show that the values of $b_i(\T_d)$ for $1 \le d \le 2^i$ determine a formula for $b_i(\T_d)$ for all $d$.
	Applying this to the case of $b_2(\T_d)$ and $b_3(\T_d)$ and using the results in Table~\ref{tab:tn} we obtain
	$
	b_2(\T_d) =\frac{1}{2}(4^d-2^d)$ and $b_3(\T_d) = \frac{1}{24}\left(  4\cdot 8^d - 3\cdot 6^d -6\cdot 4^d +5\cdot 2^d \right).
$
	For $i \geq 4$ this technique requires knowledge of $b_i(\T_d)$ for at least $1 \leq d \leq 16$.
\end{remark}

\subsection{The resonance arrangement}
The next arrangement we consider appears as a restriction of the threshold arrangement. 
\begin{definition}
\label{def:resonance}
The \mydef{resonance arrangement}  is the restriction of $\mathcal T_d$ to the hyperplane $H_{(0,\ldots,0)}$. Equivalently, for $d\geq 1$ the resonance arrangement is \vspace{-0.4em}
\[
\mydef{\mathcal R_d} : = \{\{c_1x_1+c_2x_2+\cdots+c_dx_d = 0 \} \text{ with } c_i \in \{0,1\} \text{ and not all } c_i \text{ are zero}\}.
\]
\end{definition}\vspace{-0.4em}

The chambers of the resonance arrangements are in bijection with generalized retarded functions in quantum field theory~\cite{Eva95}.
An overview of the applications of the resonance arrangement is given in~\cite[Section 1]{Kue20}.
 A formula for their number of chambers remains elusive, let alone one for their Whitney numbers. Nonetheless, partial formulas and bounds exist~\cite{BTDWW12,GMP21,Kue20,Zue92}. 

The numbers of chambers of the resonance arrangements are listed in the sequence \href{https://oeis.org/A034997}{A034997} in the OEIS up to $d=9$.
The Whitney numbers are published in~\cite{KTT11} up to $d=7$.
Our software was able to determine the Whitney numbers of $\cR_8$ and $\cR_9$ confirming the concurrent computations in~\cite{CS21}.
The computation for $\cR_9$ took ten days, running multithreaded on $42$ \texttt{Intel Xeon E7-8867 v3} CPUs. 
All Whitney numbers of $\cR_d$ up to $d=9$ are given in~\Cref{tab:rn} and the timings are listed in Table \ref{tab:OurTimings}.

\subsection{Separability arrangements of the cross-polytopes}

The \mydef{cross-polytope} of dimension $d$ is the polytope with the $2d$ vertices $\{\pm e_i\}_{i=1}^d$.
Its symmetry group is the hyperoctahedral group of order $d!2^d$.
We define the arrangement~$\mydef{\cC_d}$ in $\R^{d+1}$ to be the separability arrangement of its vertices.

Our computations show that $
|\ch(\cC_d)|=2\cdot 3^d -2^d$ for $d \leq 20$, suggesting that $|\ch(\cC_d)|$ agrees with this sequence (\href{https://oeis.org/A027649}{A027649} in the OEIS). 
This can indeed be proven by applying Athanasiadis' finite field method~\cite{Ath96} and seems to be a new result obtained through experiments with our algorithm.

\subsection{Separability arrangements of permutohedra}

The \mydef{permutohedron} of dimension $d-1$ is the convex hull of the $d!$ points $\sigma(1,\dots,d)$ for all $\sigma\in\mathfrak S_{d}$.
The separability arrangements $\mydef{\cP_d}$ of these points in $\R^{d+1}$ consist of $d!$ hyperplanes.
We record their Whitney numbers in~\Cref{tab:pn} for $1\le d \le 6$.

\subsection{Separability arrangements of demicubes}

The $d$-\mydef{demicube} is the convex hull of those vertices of the hypercube $[0,1]^d$ which have an odd number of $1$'s.
For instance, the $3$-demicube is a regular tetrahedron.
We denote by~$\mydef{\cD_d}$ the corresponding separability arrangement
consisting of $2^{d-1}$ hyperplanes in $\R^{d+1}$.
\Cref{tab:dn} contains the Whitney numbers of $\cD_d$ up to $d=9$.

\subsection{Separability arrangements of some regular polytopes}
In  \Cref{tab:reg}, we provide the Whitney numbers for the separability arrangements corresponding to the remaining two Platonic solids: the icosahedron and the dodecahedron. This table also contains the Whitney numbers of the separability arrangements of the vertices of the regular $24$-cell, $600$-cell, and  $120$-cell. 
Except for the $24$-cell, each of these computations uses irrational realizations.

\subsection{Discriminantal arrangements}
Given $n$ points in $\R^d$ in general position, the discriminantal arrangement $\mydef{\textrm{Disc}_{d,n}}$ is the hyperplane arrangement in $\R^d$ consisting of the ${n}\choose{d}$ hyperplanes spanned by $d$-subsets of such  points. This arrangement, originally called the ``geometry of circuits'' was introduced by Crapo \cite{Crapo2}.  We verify the Whitney numbers of $\textrm{Disc}_{4,n}$ for $5 \leq n \leq 16$ given in \cite[Section 4.4]{Koizumi}. From this data, we recover their formula for the characteristic polynomial of $\textrm{Disc}_{4,n}$ for all $n$. A deformation of this arrangement appears in physics \cite{cachazo2,cachazo1}
and we were able to confirm the	chamber	counts given in these papers.

\section{Timings}
\label{sec:Timings}
While other pieces of software for counting chambers of arrangements exist, they do not take advantage of symmetry and some compute significantly more data than our algorithm does. Consequently, our software outperforms them with respect to the calculation of Whitney numbers as shown below. 
\vspace{-10pt}
\begin{table}[h]
	\begin{footnotesize}
		\[\def\arraystretch{1.5}{
			\begin{array}{c| r r r r rr }			
			 \textrm{Software} & d=3 & 4 & 5 &  6 & 7 & 8\\
			\hline
			\hline
			\texttt{CountingChambers.jl} \tiny{\text{ w/ symm.}}  &0.0038 s & 0.011s& 0.035 s&  0.12s &2.89s & 19.8m \\
			\texttt{CountingChambers.jl} \tiny{\text{ w/o symm.}}  & 0.0002s & 0.0004s &  0.004s&  0.53s& 6.2m & * \\	
			\texttt{polymake}  & 0.3s& 4.31s & 3.9m & * & \\
			\texttt{sage} &  0.05s&  0.21s&  5.45s & 9.2m & * \\
						\texttt{GAP} & 0,006s &  0.035s & 1.09s & 1.9m & 12.87h & *\\ 
			\hline
			\end{array}}
		\]
	\end{footnotesize}
		\caption{Timings for computing the number of chambers  $|\ch(\cR_d)|$ of the resonance arrangement $\cR_d$ for $3\le d \le 8$ on a single thread (\texttt{Intel Core i7-8700}). An asterisk * indicates that the computation was terminated after a day.}
	\label{tab:CompareToOthers}
\end{table}
\vspace{-3em}

The implementation \cite{KP20} in \texttt{polymake} computes much more information than the Whitney numbers, namely a chamber decomposition of the arrangement. The \texttt{sage} implementation, on the other hand,  uses basic deletion and restriction as in Algorithm \ref{algo:extendedDR}. Similarly, the \texttt{GAP} package \texttt{alcove} \cite{alcove} computes the Tutte polynomial by simple deletion and restriction and then specializes this to the characteristic polynomial.

To illustrate the performance of our software on the arrangements from Section \ref{sec:Examples}, we collect our timings in Table \ref{tab:OurTimings}.
This table also shows the growth in complexity for computing the number of chambers of these~arrangements.
Based on our profiling, the main bottleneck in our implementation is the identifications of orbits. Thus, improving \texttt{pseudo\_{\hspace{0.07em}}minimal\_image} would be the most direct method for making our code faster.  \vspace{-1em}
\begin{table}[htpb]	
	\begin{scriptsize}		
	\[\def\arraystretch{1.5}{			
	\begin{array}{c| c | r r r r r r r}
		\A & |\Aut(\A)| & d=3& 4 & 5 & 6 & 7 & 8 & 9 \\
		\hline			\hline			\mathcal T_d & (d+1)!2^d& 0.005s  & 0.013s & 0.041s & 0.28s & 33.17s & 8.16h & \\			
		\mathcal R_d& (d+1)! &0.004 s & 0.011s& 0.035 s& 0.12s &2.89s & 19.8m & \sim 10d^+\\
		\mathcal C_{2d} & (2d)!2^{2d}& 0.015s & 0.039s & 0.085s& 0.183s & 0.42s & 1.158s & 4.50s\\
		\mathcal P_d & d! & 0.003s& 0.013s & 6.398s & \sim 8d^+ & &&\\			
		\mathcal D_d & (d)!2^{d-1}& 0.002 s&0.005 s& 0.018s & 0.049s &0.54s&1.9m& \sim 8d^+\\			
		\textrm{Disc}_{4,n} & n!&-&0.0003s &0.0047s &0.055s  & 0.71s &7.62s& 41.14s \\			
		\hline			\end{array}}		
		\]	
		\end{scriptsize}	\caption{Our timings on examples from Section~\ref{sec:Examples}. Computations ran on a single thread {(\texttt{Intel~Core~i7-8700})} except for $\mathcal{R}_9$
		which ran on $42$ threads (\texttt{Intel~Xeon E7-8867}).}	\label{tab:OurTimings}	\end{table}
		
		\vspace{1in}
		
\begin{table}[!htpb]
\begin{footnotesize}
		\[\def\arraystretch{1.5}{
			\begin{array}{c| r r r r r }			
			 \mathcal A& \#\text{threads}=1 & 2 & 4 &  8 & 12 \\
			\hline
			\hline
			\mathcal{R}_8  & 19.8m & 10.5m & 6.3m&  5.9m &  5.1m\\
\mathcal{T}_8  &  8.16h& 3.9h & 2.4h & 1.8h   &  1.6h \\
			\hline
			\end{array}}
		\]

\end{footnotesize}
		\caption{Comparison of the effect of number of threads on run times (\texttt{Intel Core i7-8700}).}
	\label{tab:parallel}
\end{table}

		\vspace{-1.5em}

\appendix 

\section{Tables of Whitney numbers}

\begin{table}[htpb]
	\begin{scriptsize}
		\[\def\arraystretch{1.5}{
			\begin{array}{c| r r r r r r r r}			
			d & 1 & 2 & 3 & 4 & 5 & 6 & 7 & 8  \\
			\hline
			\hline
			b_0(\T_d)				 & 1 & 1 & 1 & 1 & 1 & 1 & 1 & 1 \\
			b_1(\T_d)=|\T_d| & 2 & 4 & 8 & 16 & 32 & 64 & 128 & 256 \\
			b_2(\T_d) & 1 & 6& 28& 120& 496& 2016& 8128& 32640\\
			b_3(\T_d) &  & 3 & 44 & 460 & 4240 & 36848 & 310464 & 2569920 \\
			b_4(\T_d) &  &  &  23 & 820 & 19660 & 400400 & 7493808 & 133492800\\
			b_5(\T_d) &  &  &   &  465 & 43014 & 2453248 & 112965776 &  4626016752 \\
			b_6(\T_d) &  &  &   &  &  27129 & 7111650 &  987779688 & 103818315888 \\
			b_7(\T_d) &  &  &   &  &  &  5023907 & 4075759064 & 1382897843304 \\
			b_8(\T_d) &  &  &   &  &  &  & 3193753807 & 8676817935144 \\
			b_9(\T_d) &  &  &   &  &  &  &  & 7393243346241\\
			\hline
			|\ch(\T_d)| & 2 & 14 & 104 & 1882& 94572& 15028134& 8378070864& 17561539552946 
			\end{array}}
		\]
	\end{scriptsize}
	\caption{The values of $b_i(\T_d)$ and $|\ch(\T_d)|$ of the \mydef{threshold arrangement} for $1\le d \le 9$ and $0\le i \le d$.}
	\label{tab:tn}
\end{table}

\begin{table}[]
	\begin{scriptsize}
		\[\def\arraystretch{1.5}{
			\begin{array}{c| r r r r r r r r r}
			d & 1 & 2 & 3 & 4 & 5 & 6 & 7 & 8  &9\\
			\hline
			\hline
			b_0(\cR_d)				 & 1 & 1 & 1 & 1 & 1 & 1 & 1 & 1 & 1 \\
			b_1(\cR_d)=|\cR_d| & 1 & 3 & 7 & 15 & 31 		& 63 & 127 & 255 & 511 \\
			b_2(\cR_d) &  				 & 2 & 15& 80& 375		& 1652& 7035& 29360 & 120975\\
			b_3(\cR_d) &  				 &  & 	9   & 170 & 2130 & 22435 & 215439 & 1957200 & 17153460\\
			b_4(\cR_d) &  				 &  &      & 104 & 5270 & 159460 & 3831835 & 81029004 & 1582492380\\
			b_5(\cR_d) &  				&  &   &   			& 3485 & 510524 & 37769977 &  2076831708 &  96834110730\\
			b_6(\cR_d) &  				&  &   &  &   					& 371909 &  169824305 & 30623870732 &  3829831100340\\
			b_7(\cR_d) &   			&  &   &  &  &   									& 135677633 & 207507589302 & 89702833260450\\
			b_8(\cR_d) &  				&  &   &  &  &  &  														& 178881449368 & 973784079284874\\
			b_9(\cR_d) &  				&  &   &  &  &  &  														&  & 887815808473419\\
			\hline
			|\ch(\cR_d)| & 2 & 6 & 32 & 370& 11292& 1066044 & 347326352& 419172756930 & 1955230985997140
			\end{array}}
		\]
	\end{scriptsize}
	\caption{The values of $b_i(\cR_d)$ and $|\ch(\cR_d)|$ of the \mydef{resonance arrangement} for $1\le d \le 9$ and $0\le i \le d$.
	We submitted these Whitney numbers to the OEIS as the sequence \href{https://oeis.org/A344494}{A344494}.}
	\label{tab:rn}
\end{table}

\begin{table}[h]
	\begin{scriptsize}
		\[\def\arraystretch{1.5}{
			\begin{array}{c| r r r r r r r r}			
			d   								&2 & 3 & 4 & 5 & 6 & 7 & 8  & 9\\
			\hline
			\hline
			b_0(\cD_d)				     & 1 & 1 & 1    & 1       & 1 			& 1 			    & 1 & 1 \\
			b_1(\cD_d)=|\cD_d|     & 2 & 4 & 8   & 16     & 32 			& 64 			  & 128 & 256 \\
			b_2(\cD_d)  				 &  1 & 6 & 28 & 120   & 496		& 2016			& 8128& 32640\\
			b_3(\cD_d) 					&  	 0 & 4& 50 & 500   & 4480      & 38304 		  & 319200 & 2622400 \\
			b_4(\cD_d) 					&  	   &  1&  44 & 1160 & 24340   & 461496 		& 8283744 & 143504320\\
			b_5(\cD_d) 					&  		&  &   15&  1362 & 76364   & 3486448 	& 143595816 &  5483536464 \\
			b_6(\cD_d) 					&  		&  &   		&  597  & 120942 & 15440376  & 1615624080 & 145378334304 \\
			b_7(\cD_d) 					&       &  &   		&          & 64903   &  33803416 & 10878083096 & 2574289938400\\
			b_8(\cD_d) 					&  	    &  &   		&          &  			  &  21424343 & 35828091880	& 27816202212040\\
			b_9(\cD_d) 					&  	    &  &   		&          &  			  &  				 & 26430009593	& 146101801794362 \\
			b_{10}(\cD_d) 			   &  	    &  &   		&          &  			 &   				 & 							& 120719853808577\\
			\hline
			|\ch(\cD_d)| 				& 4  & 16& 146 & 3756& 291558& 74656464& 74904015666& 297363155783764
			\end{array}}
		\]
	\end{scriptsize}
	\caption{The values of $b_i(\cD_d)$ and $|\ch(\cD_d)|$ of the \mydef{demicube arrangement} for $2\le d \le 9$ and $0\le i \le d+1$.}
	\label{tab:dn}
\end{table}

\begin{table}[h]
	\begin{scriptsize}
		\[\def\arraystretch{1.5}{
			\begin{array}{c| r r r r r r}			
			d   								&1 & 2 & 3 & 4 & 5 & 6  \\
			\hline
			\hline
			b_0(\cP_d)				     & 1 & 1 & 1    & 1        & 1 			 	& 1 			     \\
			b_1(\cP_d)=|\cP_d|     & 1 & 2 & 6    & 24      & 120 	   		& 720			   \\
			b_2(\cP_d)  				 &    & 1 & 15  & 276    & 7140			& 258840		\\
			b_3(\cP_d) 					&  	  &    & 10  & 1423   & 246605      & 59577390		   \\
			b_4(\cP_d) 					&  	   &    &     & 1170  & 4290610   & 9271534305	\\
			b_5(\cP_d) 					&  		&  &       &   			& 4051026   & 834595018036 	 \\
			b_6(\cP_d) 					&  		&  &   		&          &  					& 825382803000  \\
			\hline
			|\ch(\cP_d)| 				& 2  & 4& 32 & 2894& 8595502& 1669309192292
			\end{array}}
		\]
	\end{scriptsize}
	\caption{The values of $b_i(\cP_d)$ and $|\ch(\cP_d)|$ of the \mydef{permutohedron arrangement} for $1\le d \le 6$ and $0\le i \le d$.}
	\label{tab:pn}
\end{table}

\begin{table}[h]
	\begin{scriptsize}
		\[\def\arraystretch{1.5}{
			\begin{array}{c| r r |r r r r r r}			
			\mbox{polytope} 	&\mbox{Icosahedron} &\mbox{Dodecahedron} 	&24\mbox{-cell} & 600\mbox{-cell} & 120\mbox{-cell} \\
			\hline
			\hline
			b_0(\A)		  &1 & 1		     & 1 & 1 				& 1    			  \\
			b_1(\A)=|\A|    & 12& 20   & 24 & 120 			& 600    			   \\
			b_2(\A)  	  & 66& 166			 &  276  & 7140 			& 179700 	\\
			b_3(\A) 	  & 157&577 				&  	 1630 &   225782 	& 31972550		   \\
			b_4(\A) 	  & 102& 430				&  	4308   &   3118740 &    2979870540	\\
			b_5(\A) 	  & - & -  				&  	2931	&  2899979 &     2948077091	 \\
			\hline
			|\ch(\A)| 		  &338 & 1194		& 9170  & 6251762 & 5960100482& &
			\end{array}}
		\]
	\end{scriptsize}
	\caption{The values of $b_i(\A)$ and $|\ch(\A)|$ of the \mydef{icosahedral and dodecahedral arrangements} as well as \mydef{arrangements stemming from regular $4$-polytopes} for  and $0\le i \le 5$.}
	\label{tab:reg}
\end{table}

 \clearpage

\bibliographystyle{abbrv}
\bibliography{arrangements.bib}

\end{document}